\documentclass[12pt]{amsart}
\usepackage{amsmath, amsthm, amssymb, amscd, mathptmx}

\usepackage{amsfonts}

\usepackage{color}
\usepackage{pstricks}

\usepackage[all]{xy}\CompileMatrices\SelectTips{cm}{12}

\usepackage{tikz}
\usepackage{dsfont}
\usetikzlibrary{matrix}
\usepackage[margin=1.2in]{geometry}
\usetikzlibrary[positioning,patterns]

\theoremstyle{plain}
\newtheorem{Thm}{\sc Theorem}[section]
\newtheorem{Theorem}[Thm]{\sc Theorem}
\newtheorem{Corollary}[Thm]{\sc Corollary}
\newtheorem*{Corollary*}{\sc Corollary}

\newtheorem{Proposition}[Thm]{\sc Proposition}
\newtheorem*{Proposition*}{\sc Proposition}
\newtheorem{Lemma}[Thm]{\sc Lemma}
\newtheorem{Conjecture}[Thm]{\sc Conjecture}

\theoremstyle{definition}
\newtheorem{Definition}[Thm]{Definition}
\newtheorem{question}[Thm]{Question}
\theoremstyle{remark}
\newtheorem{Remark}[Thm]{Remark}

\newtheorem*{Example*}{Example}
\newtheorem*{Remark*}{Remark}

\newcommand{\CC}{{\mathbb C}}

\newcommand{\ZZ}{{\mathbb Z}}

\newcommand{\PP}{{\mathbb P}}
\newcommand{\QQ}{{\mathbb Q}}
\newcommand{\RR}{{\mathbb R}}

\newcommand{\cE}{{\mathcal E}}
\newcommand{\cF}{{\mathcal F}}
\newcommand{\cG}{{\mathcal G}}
\newcommand{\cH}{{\mathcal H}}
\newcommand{\cL}{{\mathcal L}}

\newcommand{\cJ}{{\mathcal J}}

\newcommand{\cM}{{\mathcal M}}
\newcommand{\cN}{{\mathcal N}}
\newcommand{\cO}{{\mathcal O}}

\newcommand{\GL}{\mathop{\rm GL\, }}
\newcommand{\Sing}{\mathop{\rm Sing\, }}

\newcommand{\Hom}{{\mathop{{\rm Hom}}}}

\newcommand{\torsion}{\mathop{\rm torsion}}

\newcommand\Supp{\mathop{\rm Supp\, }}

\begin{document}

\title{On birational boundedness of foliated surfaces}
\author[C.~Hacon]{Christopher D. Hacon}
\address{
Department of Mathematics, University of Utah, 155 S 1400 E, JWB 233,
Salt Lake City, UT 84112, USA}
\email{hacon@math.utah.edu}

\author[A.~Langer]{Adrian Langer}
\address{Institute of Mathematics, University of Warsaw,
ul.\ Banacha 2, 02-097 Warszawa, Poland}
\email{alan@mimuw.edu.pl}

\thanks{The first author was partially supported by NSF research grants no: DMS-1801851,  DMS-1840190
and by a grant from the Simons Foundation; Award Number: 256202. The second author was partially supported Narodowe Centrum Nauki, contract number 2018/29/B/ST1/01232.}
\date{\today}

\maketitle


\medskip

\begin{abstract}In this paper we prove a result on the effective generation of pluri-canonical linear systems on foliated surfaces of general type. Fix a function $P: \mathbb Z _{\geq 0}\to \mathbb Z $, then there exists an integer $N>0$ such that  if $(X,\cF)$ is a canonical or nef model of a foliation of general type with Hilbert polynomial $\chi (X, \cO _X(mK_\cF))=P(m)$ for all $m\in \mathbb Z _{\geq 0}$, then $|mK_\cF|$ defines a birational map for all $m\geq N$.

On the way, we also prove a Grauert-Riemenschneider type vanishing theorem for foliated surfaces with canonical singularities. 
 \end{abstract}

\section*{Introduction}
In recent years a powerful theory for the birational classification of foliated algebraic surfaces has been developed by Brunella, McQuillan and others (see \cite{McQ08}, \cite{Br15} and references therein). This theory extends the classical results 
of the birational classification of algebraic surfaces in terms of their canonical bundle to the case of foliated surfaces in terms of the canonical bundle of their foliation $K_\cF$. This classification is
particularly precise for foliated surfaces of non-maximal Kodaira dimension $\kappa (K_\cF)<2$. In the case of maximal Kodaira dimension $\kappa (K_\cF)=2$, by work of Brunella and McQuillan, it is known that smooth foliated surfaces with canonical singularities admit unique minimal, nef and canonical models. In light of the existence of canonical models, one may even hope that there is a well behaved moduli functor for these canonical models of general type. Note however that by a result of McQuillan, for any canonical model $(X,\cF)$ with cusp singularities, $K_{\cF}$ is not a $\mathbb Q$-Cartier divisor. In particular, the canonical rings $R(K_\cF)$ of canonical models $(X,\cF)$ with cusp singularities are not finitely generated or equivalently $K_\cF$ is not ample.  Therefore, if such a moduli functor exists, it is expected not to be algebraic.

With a view to further understanding the birational geometry of foliated surfaces of general type and in particular issues related to the existence of a moduli functor, the first most natural question to address is the boundedness of this functor. To this end we ask the following
\begin{Conjecture} For any  integer valued function  $P:\mathbb Z_{\geq 0}\to \mathbb Z$, does there exist an integer $m_P$ such that if $(X,\cF)$ is a canonical model of a complete foliated surface with   ${\rm kod} (\cF)=2$ and $\chi (X,\cO _X(mK_\cF) )=P(m)$ for all $m\geq 0$, then  for all $ m>0$ divisible by $ m_P$, $|mK_\cF|$ defines a birational map which is an isomorphism on the complement of the cusp singularities?
\end{Conjecture}
Notice that as observed above, $mK_\cF$ is not Cartier at the cusp singularities and hence these singularities are necessarily contained in the base loci of $|mK_\cF|$ for all integers $m>0$. In this paper, we prove an important first step towards this conjecture (see Theorems \ref{birational-boundedness-nef} and \ref{birational-boundedness-canonical}).
\begin{Theorem}\label{t-main}  For any  integer valued function  $P:\mathbb Z_{\geq 0}\to \mathbb Z$, there exists an integer $m_P$ such that if $(X,\cF)$ is a canonical or a weak nef model (see Definition \ref{d-wnm}) of a complete foliated surface with   ${\rm kod} (\cF)=2$ and $\chi (X,\cO _X(mK_\cF ) )=P(m)$ for all $m\geq 0$, then  for all $ m\ge  m_P$, $|mK_\cF|$ defines a birational map.
\end{Theorem}
As an immediate consequence, following \cite{Per02}, we have
\begin{Corollary} For any  integer valued function  $P:\mathbb Z_{\geq 0}\to \mathbb Z$ and any integer $g\geq 0$, there exists an integer $d>0$ such that if $(X,\cF )$ is  a weak nef model of a complete foliated surface with   ${\rm kod} (\cF)=2$ and $\chi (X,\cO _X(mK_\cF ))=P(m)$ for all $m\geq 0$, and  if $(X,\cF)$ has a meromorphic first integral whose general fiber has geometric genus $g$, then the general leaf $C$ has bounded degree $C\cdot K_\cF\leq d$.
\end{Corollary}
\begin{proof} The proof is identical to the one in \cite{Per02}. We include a sketch for the convenience of the reader. Let $f:X'\to X$ be a resolution such that if $\cF'=f^*\cF$, then 
there is a morphism to a curve $g:X'\to B$ where $T_{\cF'}={\rm ker}\left( T_{X'}\to g^*T_B \right)$.
Notice that $K_{\cF'}|_{C'}=K_{C'}$ and we have a surjection $\cO _{X'} (mK_{\cF'})\to
\cO _{C'}(mK_{C'})$, where $C'$ is a general fiber of $g$. Since $h^0(mK_{C'})=(2m-1)(g-1)$ for $m\geq 2$ and $h^0(mK_{\cF'})=\frac{{\rm vol}(K_{\cF})}2m^2+O(m)$, it follows easily that there exists an integer $m_0$ (depending only on $P$ and $g$) such that $|m_0K_{\cF'}-C'|$ is non-empty. Let $C=f(C')$.
Then $|m_0K_{\cF}-C|$ is also non-empty and as $K_{\cF}$ is nef, we have $K_{\cF}\cdot C\leq m_0K_{\cF}^2$.
\end{proof}
We remark that even though the hypothesis $\chi (X,\cO _X(mK_\cF ))=P(m)$ for all $m\geq 0$ is very natural from the point of view of moduli spaces, one could hope that (analogously to the case of SLC models  cf. \cite{HMX18}), the behaviour of pluricanonical maps is determined simply by the volume ${\rm vol}(K_\cF)$. 
It would also be interesting to understand the structure of the set of canonical volumes. The most natural question is.
\begin{question} Let $V=\{ {\rm vol}(K_\cF)\}$ where $(X,\cF)$ are canonical  models of  foliated surfaces of general type. Is $V$ well ordered and in particular does it admit a positive minimum?
\end{question}
Next we recall an example (which was communicated to us by F. Bernasconi \cite{Ber19}) which shows that the set $V$ is not discrete and in fact it has accumulation points from below.

\vskip.3cm
\noindent
{\bf Jouanolou's foliation.} Let $\cJ$ be the Jouanolou's foliation on $\mathbb P ^2$ defined by the vector field on $\mathbb C ^3\setminus \{0 \}$ given by \[ v=Z^d\frac \partial {\partial X}+
X^d\frac \partial {\partial Y}+Y^d\frac \partial {\partial Z}\]
A direct computation shows that $\cJ _d$ has reduced singularities, $K_{\cJ _d}=\cO _{\mathbb P ^2}(d-1)$ and the automorphism group of the foliation is the following (see \cite[p.~160-162]{Jo79})
\[ {\rm Aut}(\cJ _d)= \mathbb Z/(d^2+d+1)\mathbb Z \rtimes \mathbb Z/3\mathbb Z\]
and is generated by 
\[ T([X:Y:Z])=[ Y:Z:X],\qquad \psi ([X:Y:Z])=[X:\zeta Y:\zeta ^{d+1}Z]\]
where $\zeta $ is a primitive $(d^2+d+1)$-th root of unity. For $d\geq 2$, $\cJ _d$ is a foliation on $\mathbb P ^2$ with ample canonical class and we consider the {quotient of $\mathbb P ^2$ by the cyclic group generated by $\psi$} \[f_d:(\mathbb P ^2,\cJ_d )\to (X_d, \cF _d).\] The foliation $\cF _d$ is singular with reduced singularities at $f_d({\rm Sing}(\cJ _d))$ and with terminal singularities at the points $f_d([1:0:0])$,   $f_d([0:1:0])$, and  $f_d([0:0:1])$.
Since $K_{\cJ _d}=f_d^*K_{\cF _d}$, the foliation $\cF_d$ has ample canonical class and \[ 1> K_{\cF _d}^2=\frac {(d-1)^2}{d^2+d+1}\geq \frac 1 7,\ {\rm and}\qquad \lim_{m\to \infty } K_{\cF _d}^2=1.\] 
In particular we see that the set $V$ is not discrete and in fact $1$ is an accumulation point from below.

Finally, we prove a Grauert-Riemannschneider type vanishing theorem for foliated surfaces with canonical singularities (see Theorems \ref{strong-vanishing} and \ref{crepant-vanishing}).
\begin{Theorem} Let $f:(X,\cF )\to (Y,\cG )$ be a proper birational morphism of foliated surfaces with only canonical singularities, then 
\begin{enumerate}
\item $f_*\cO _X(mK_\cF)=\cO _Y(mK_\cG) $ for all $m\geq 0$,
\item $R^1f_* \cO _X(K_\cF )=0$, and
\item if $K_\cF =f^*K_\cG$, then  $R^1f_* \cO _X(mK_\cF )=0$ for all $m\ne 0$.\end{enumerate}
\end{Theorem} 
As a consequence we see that the Hilbert function $\chi (X,\cO _X(mK_\cF))$ of the canonical model determines the Hilbert function of any almost minimal model (in the same birational class).
It is natural to ask if similar results hold in higher dimensions. In particular we ask:
\begin{question} Let $f:(X,\cF )\to (Y,\cG )$ be a proper birational morphism of foliated varieties with only canonical singularities, then does $R^if_*  \cO _X(K_\cF )$ vanish for all $i>0$?
\end{question}
\medskip

\subsection*{Notation.}

All varieties and spaces in this paper are  defined over the complex numbers.
Unless stated otherwise, by a surface we mean a $2$-dimensional algebraic space. {We will work exclusively with complete normal algebraic spaces of finite type over $\CC$ and with their localizations at closed points $x\in X$. Occasionaly, we will consider the analytic germ of an algebraic space at closed points $x\in X$.}

\section{Preliminaries}

\subsection{Normal surfaces}

In this subsection we recall several basic results on normal surfaces that will be used throughout the paper.
When considering canonical models of foliations we will necessarily need to work with algebraic spaces. However,
in many cases such surfaces will be projective due to the following basic result of Artin (see \cite[Theorem 2.3]{Ar62}).

\begin{Theorem} \label{Artin}
Let $X$ be a normal complete surface with at most rational singularities. Then $X$ is projective.
\end{Theorem}

\begin{proof}
For any $X$ as above there exists a proper birational morphism $Y\to X$ from a smooth projective surface $Y$ {(see \cite[Corollary 3.43]{Kol07})}. 
So the assertion follows from the last part of  \cite[Theorem 2.3]{Ar62}.
\end{proof}
{\begin{Corollary}\label{c-proj}
Let $(X,\cF)$ be a complete foliated surface with canonical singularities such that $K_{\cF}$ is $\QQ$-Cartier, then $X$ is projective.\end{Corollary}
\begin{proof} By McQuillan's classification of canonical singularities (see \cite{McQ08} or the discussion in \S \ref{contributions}), we know that all such singularities are rational. The statement now follows from Theorem \ref{Artin}.
\end{proof}}
\subsubsection{Intersection theory on normal surfaces}\label{Mumford}
  
Here we recall Mumford's intersection pairing on normal surfaces (see, e.g., \cite[Section 1]{Sa84})
and some of its basic properties. 

Let $Y$ be a normal complete surface. Let $f: X\to Y$ is a proper birational morphism  from a smooth  surface $X$ and let $C=\sum _i C_i$ be the exceptional  divisor of $f$. If $D$ is a Weil $\RR$-divisor on $Y$ then we define $f^*D$ as $f_*^{-1}D+\sum x_i C_i$, where $f_*^{-1}D$ is the strict transform of $D$ and $x_i$ are the unique real numbers such that $(f_*^{-1}D+\sum x_i C_i)\cdot C_j=0$ for all $j$. If $D_1$ and $D_2$ are Weil $\RR$-divisors on $Y$ we define their intersection number by $D_1\cdot D_2= (f^*D_1)\cdot (f^*D_2)$. In case $D_1$ is a Cartier divisor and $D_2$ is a curve, $D_1\cdot D_2$ agrees with the degree of the line bundle $\cO_{D_2}(D_1)$.

We say that $D_1$ and $D_2$ are \emph{numerically equivalent} and write $D_1\equiv D_2$ if for every Weil divisor $C$ we have $D_1\cdot C=D_2\cdot C$. It is sufficient to check this equality in case $C$ is an irreducible curve.

Note that once we define the intersection pairing on {(complete)} normal surfaces we can also define the pull-back $f^*D$ for any proper birational morphism $f$ of normal surfaces and any Weil $\RR$-divisor $D$ (see \cite[Section 6]{Sa84}).

{For any pseudo-effective $\RR$-divisor $D$ on a complete normal surface, we define its Zariski decomposition $D=P+N$ where $P,N$ are the unique $\RR$ divisors satisfying 
\begin{enumerate}
\item $P$ is nef,
\item $N=\sum _{i=1}^kn_iN_i$ where $n_i>0$, and the $N_i$ are prime divisors such that the intersecion matrix $(N_i\cdot N_j)$ is negative definite, and
\item $P\cdot N_i=0$ for $1\leq i\leq k$.\end{enumerate}
The existence of the Zariski decomposition is standard for projective surfaces. In general, one can prove the existence of a Zariski decomposition following the proof given in \cite[\S7]{Sa84} for complex analytic surfaces. Alternatively one can argue as follows. Let $f:X\to Y$ be a resolution such that $X$ is  projective and let $P'+N'$ be the Zariski decomposition of $f^*D$. We set $P=f_*P'$ and $N=f_*N'$.
Since $P$ is nef, it is easy to see that $P'$ is nef and $P'-f^*P=-E$ where $E$ is an effective, exceptional $\QQ$-divisor. Let $N=f_* N'$, then $N'=f^*N+E$. Since the intersection matrix of $N'$ is negative definite, it follows that $E^2<0$ unless $E\equiv 0$. Since $E\geq 0$ and $-E$ is relatively nef, we have $E=0$. Thus $N'=f^*N$ and it follows easily that the intersection matrix of $N$ is negative definite. Finally, for any component $N_i$ of $N$, we have $P\cdot N_i=f^*P\cdot f^*N_i=P'\cdot f^*N_i=0$ since $f^*N_i$ is a sum of components of $N'$ and these intersect $P'=f^*P$ trivially.}

We need the following version of the Hodge index theorem.

\begin{Lemma} \label{HIT}
If $D_1$ and $D_2$ are Weil $\RR$-divisors on $Y$ such that $(a_1D_1+a_2D_2)^2>0$ for some $a_1, a_2\in \RR$, then 
$$D_1^2\, D_2^2\le (D_1\cdot D_2)^2$$
with equality if and only if some nonzero linear combination of $D_1$ and $D_2$ is numerically equivalent to $0$.
\end{Lemma}

\begin{proof}
The assertion is well known in case $Y$ is smooth (see, e.g., \cite[D.2.2]{Re93}).
In general, the assertion follows immediately by passing to the resolution of singularities.
\end{proof}

\begin{Lemma} \label{equality} 
{Let $Y$ be the localization algebraic surface at a closed point $y\in Y$.}
Let $f: (\tilde Y, C)\to (Y,y)$ be a resolution of a rational surface singularity and let $L_1$ and $L_2$ be line bundles  on $\tilde Y$. If for every irreducible component $C_i$ of $C$ we have $L_1\cdot C_i=L_2\cdot C_i$, then $L_1$ and $L_2$
are isomorphic. In particular, we have $(f_*L_1)^{**}\simeq (f_*L_2)^{**}$.
\end{Lemma}

\begin{proof}
Our assumptions imply that $(L_1\otimes L_2^{-1})\cdot C_i=0$ for every irreducible component $C_i$ of $C$. So by 
 \cite[Corollary 2.6]{Ar62} we have $L_1\otimes L_2^{-1}\simeq \cO_{\tilde Y}$ (the assumptions of this corollary are satisfied, as the condition (a) of  \cite[Theorem 2.3]{Ar62} holds trivially for rational singularities).
\end{proof}

\subsubsection{Cyclic quotient singularities}\label{cyclic}

Let $(Y, y)$ be a cyclic quotient singularity of type  $\frac{1}{n}(1,q)$ for some relatively prime positive integers $n\ge 2$ and $1\le q<n$. So locally analytically $(Y, y)=(\CC ^2, 0)/G$ and $G=\left\langle \left(
\begin{array}{cc}
\epsilon&0\\
0&\epsilon^q\\
\end{array}
\right)\right\rangle $, where $\epsilon$ is a primitive $n$-th root of $1$. Let $\mu : \CC^2\to Y$ be the quotient map
and let $f: (X,C)\to (Y,y)$ be the minimal resolution of $(Y,y)$. The exceptional divisor $C=\bigcup _{j=1}^r C_j$
is a Hirzebruch--Jung string, i.e., it consists of smooth rational curves $C_i$ such that $C_i^2\le -2$, $C_i\cdot C_j=1$
if $|i-j|=1$ and $C_i\cdot C_j=0$ if $|i-j|>1$. The irreducible representations of $G$ are given by the characters $\chi_i$ defined by sending the chosen generator of $G$ to $\epsilon ^i$ for $i=0,...,n-1$. 
Each character $\chi_i: G\to \CC^*=\GL (1,V)$ gives rise to a reflexive sheaf $\cN_i=(\mu _*\cO_{\CC^2}\otimes_{\CC} V )^G$, where $G$ acts on $\mu _*\cO_{\CC^2}\otimes  _{\CC} V $ by $g (f\otimes v)= g(f)\otimes (\chi_i (g)) (v)$. Another choice is to consider sheaves $\cL_i$ defined as the $i$-th eigensheaves of the action of $G$ on $\mu_*\cO_{\CC^2}$ (cf. \cite[8.3]{Re85}), i.e., $\cL_i$ is a subsheaf of
$\mu_*\cO_{\CC^2}$ formed by sections $f$ on which the generator of $G$ acts by multiplication by $\epsilon ^i$. 
These two choices are related by equality $\cN_i=\cL_{n-i}=\cL _i^*$.

Let us write the continued fraction expansion 
$$\frac{n}{q}=b_1-\frac{1}{b_2-\frac{1}{\cdots -\frac{1}{b_r}}},$$
where $b_i$ are integers $\ge 2$. It is well known that $C_j^2=-b_j$ for $j=1,...,r$.

The following theorem is the main result of \cite{Wu85}.

\begin{Theorem}\label{Wunram}
Let $s_0,...,s_r$ be positive integers defined recursively by
$s_0:=n$, $s_1=q$ and $s_{j}:=b_{j-1}s_{j-1}-s_{j-2}$
for  $2\le j\le r$. For every $0\le i\le n-1$ there exist uniquely defined non-negative integers $d_1$,..., $d_r$ such that
\begin{align*}
i=d_1s_1+t_1,& \quad 0\le t_1<s_1,\\
t_j=d_{j+1}s_{j+1} +t_{j+1}, &\quad 0\le t_{j+1}<s_{j+1}, \quad 1\le j\le r-1
\end{align*}
Let $\cM_i=f^*\cN_i/\torsion$. Then $\cM_i$ is a line bundle with $\cM_i\cdot C_j=d_j$ for  $j=1$, ..., $r$.
\end{Theorem}

 \subsubsection{Riemann--Roch theorem on normal surfaces}\label{s-RR}

 \begin{Theorem}\label{RR}
 Let $Y$ be a normal complete surface and let $D$ be a Weil divisor on $Y$. Then there is a formula
 $$\chi(Y, \cO _Y(D))=\frac{1}{2} (D^2-K_Y\cdot D)+\chi (Y,\cO_Y)+\sum _{y\in \Sing Y}a(y, D),$$
 where $a(y, D)$ is a local contribution of $\cO_Y (D)$ at $y$ depending only on the local isomorphism class of the reflexive sheaf $\cO_Y(D)$ at $y$.
 \end{Theorem}
 
In case $Y$ has only quotient singularities this follows from \cite[Corollary 8.6]{Re85}. In general, the theorem follows from \cite[Section 3]{La00}.  Although \cite{La00} assumes projectivity, none of the proofs use this assumption and the above theorem holds more generally for normal complete surfaces. Nevertheless, in all the cases that we use the above theorem the surface is in fact projective.
Let us recall that $a(y, D)$ can be computed in the following way. Let $f: (X, C)\to (Y,y)$ be any resolution of the normal surface singularity and let $\tilde D$ be any divisor such that $f_*\tilde D=D$.  Let $c_1(y, \tilde D)$ be the unique $\QQ$-divisor supported on the exceptional locus $C$ of $f$ such that $c_1(y,\tilde D)\cdot C_i=\deg \cO_{C_i}(\tilde D)$ for all exceptional curves $C_i$. {Let us set 
$$\chi (y, \cO_X (\tilde D))= \dim \, (\cO_Y(D)/f_*\cO_X(\tilde D))_y +\dim \, 
(R^1f_*\cO_X({\tilde D}))_y.$$
By \cite[Definition 2.7]{La00} and \cite[Proposition 2.8]{La00} we have
\begin{equation}\label{contribution-formula}
a(y, D)=\frac{1}{2}c_1(y,\tilde D)(c_1(y,\tilde D) -c_1(y, K_{X}))+ \chi (y, \cO_X (\tilde D))- \dim \, 
(R^1f_*\cO_X)_y.
\end{equation}}
In particular, if $D$ is a Cartier divisor at $y$ then $a(y, D)=0$.

\medskip

In case $(Y,y)$ is  a cyclic quotient singularity of type  $\frac{1}{n}(1,q)$, the local contributions $a(y, \cL_i)$ for reflexive sheaves of type $\cL_i$ were computed in  \cite[Example 5.6]{La00} and for $0\le i<n$ we have
$$a(y, \cL_i)=  \frac{1}{n}\left( \sum _{j=0}^{i-1} \overline{c j} -\frac{i (n-1)}{2}\right),$$
where $c$ denotes an integer such that $qc \equiv -1\mod n$ and $\overline x$ denotes the remainder from dividing $x$ by $n$. Let us remark that $\omega_Y$ is of type $\cL_{q+1}$, so by \cite[Proposition 2.10]{La00} we have
$$a(y, \cL_q) =a(y, \cL_1)=  -\frac{n-1}{2n}.$$

\subsubsection{Modified Euler characteristic}\label{mod-Euler}

Let $f: X\to Y$ be a proper birational morphism between normal complete surfaces.
Similarly as in Subsection \ref{s-RR}, for  a Weil divisor $\tilde D$ on $X$ and a point $y\in Y$ we can consider
$$\chi  (y, \cO_X (\tilde D))= \dim \, (\cO_Y(f_*\tilde D)/f_*\cO_X(\tilde D))_y +\dim \, 
(R^1f_*\cO_X({\tilde D}))_y.$$
For simplicity we omit $f$ in the notation as it is implicitly contained in the fact that $\tilde D$ is a divisor on $X$.
We also set 
$$\chi   (f, \cO_X (\tilde D))=\sum _{y\in Y}\chi (y,  \cO_X (\tilde D)).$$
Then the Leray spectral sequence implies the equality
$$\chi (X, \cO_X (\tilde D))=\chi (Y, \cO_Y (f_*\tilde D))-\chi   (f, \cO_X (\tilde D)).$$
Let  $f: X\to Y$ and $g:Y\to Z$ be proper birational morphisms between normal complete surfaces.
Applying the above equality to $f$, $g$ and $g\circ f$ we obtain
$$\chi (g\circ f, \cO_X (\tilde D))=\chi   (f, \cO_X (\tilde D)) + \chi (g, \cO_Y (f_*\tilde D)).$$
This equality has an obvious local analogue that  can be proven using the Leray spectral sequence:
{$$\chi (z, \cO_X (\tilde D))= \chi (z,  \cO_Y (f_*\tilde D))+\sum _{y\in f^{-1}(z)} \chi (y, \cO_X (\tilde D)) .$$}

\subsubsection{Adjoint linear systems on normal surfaces}\label{adjoint}

Let us recall the following special case of \cite[Theorem 0.1, (0.3.2), and Remark (3), p.~60]{La01}.

\begin{Theorem} \label{Reider-type}
Let $L$ be a pseudoeffective divisor on a normal complex projective surface $Y$. Let $L=P+N$ be the Zariski decomposition of $L$ and let $\zeta$ be a $0$-dimensional subscheme of $Y$ contained in its smooth locus. 
If $P^2>4 \deg {\zeta}$ and the restriction $H^0(Y, \cO_Y(K_Y+L))\to \cO_{\zeta} (K_Y+L)$
is not surjective then there exists a curve $C$ containing $\zeta$ and such that 
$P\cdot C\le 2 \deg {\zeta }$.
\end{Theorem}

Let us remark that the proof of the above theorem uses existence of ample divisors on $Y$, so it is not sufficient to assume that $Y$ is a normal complete surface.

\subsection{Birational geometry of foliated surfaces}

We refer the reader to \cite{McQ08} and \cite{Br15} for a detailed account of results on the birational geometry of foliated surfaces. Unfortunately,  \cite{Br15}  deals only with smooth surfaces and  \cite{McQ08} does not contain the definitions that would suit our presentation, so we collect a few of the notations, definitions and results that will be most important for us.

A foliation on a normal surface $X$ is a rank $1$ saturated subsheaf $T _\cF$ of the tangent sheaf $T_X$. 
A singular point of a foliation is either a singular point of $X$ or a point at which the quotient $T_X/T_{\cF}$ is not locally free. 
Note that our definition implies that $\cF$ has only isolated singularities.

A foliated surface is a pair $(X, \cF)$ consisting of a normal surface $X$ and a foliation $\cF$.

Note that $T_X$ is reflexive as it is isomorphic to $\Hom _{\cO_X}(\Omega _X, \cO_X)$. Therefore $T_{\cF}$ is also reflexive and we can define the canonical divisor $K_{\cF}$ of the foliation as  a Weil divisor on $X$ satisfying $\cO_X(-K_{\cF})\simeq T_{\cF}$. In particular, we have $\cO_X(K_{\cF})\simeq T^*_{\cF}$.

If $f:Y\to X$ is a proper birational morphism of normal surfaces and $\cF$ is a foliation on $X$  then we can define 
the pull-back foliation  $f^*\cF$ as follows: let $U\subset X$ be the largest open subset such that $V:=f^{-1}(U)\to U$ is an isomorphism and set $f^*\cF$ to be the unique saturated subsheaf of $T_Y$ such that $(f^*\cF)|_V\cong \cF|_U\subset T_U$. 
If $\cG$ is a foliation on $Y$ then we can also consider the push-forward foliation $f_*\cG$ by taking the saturation of the image of the composition
$$f_* T_{\cG}\to f_*T_{Y}\to (f_*T_{Y})^{**}=T_X.$$
Let us note that $f^*f_*\cG=\cG$ and $f_*f^*\cF=\cF$.
These equalities follow from  the following easy lemma.

\begin{Lemma}
Let $X$ be a {normal} irreducible algebraic space of finite type over some field and let $\eta$ be the generic point  of $X$.
Let $\cF_1$ and $\cF_2$ be saturated subsheaves of a torsion free coherent sheaf $\cE$ on $X$. If 
$(\cF_1)_{\eta}=(\cF_2)_{\eta}\subset \cE _{\eta}$ then $\cF_1=\cF_2\subset \cE$.
\end{Lemma}

\begin{proof}
By assumption the canonical map $\cF_1\to \cE\to \cE/\cF_2$ is zero at the generic point $\eta$. Since both $\cF_1$ and $\cE/\cF_2$
are torsion free, this map is zero everywhere. Therefore $\cF_1\subset \cF_2\subset \cE$. Similarly,  $\cF_2\subset \cF_1\subset \cE$, which implies the required assertion.
\end{proof}

The following definition follows  \cite[I.1.2 and Fact I.2.15]{McQ08}.

\begin{Definition}  
Let $(X,\cF)$ be a foliated normal surface and let  $f:Y\to X$ be a proper birational morphism. For any divisor $E$ on $Y$ we define the {\it discrepancy of $\cF$ along $E$}
\[ a_E(\cF)={\rm ord}_E(K_{f^*\cF } -f^*K_\cF),\]
where $f^*K_{\cF}$ is defined as in \S \ref{Mumford}.
We say that $x$ is a {\it  canonical point} (resp. a {\it terminal point}) of $(X,\cF)$ if  $a_E(\cF)\geq 0$ (resp. $a_E(\cF)> 0$) for every divisor $E$ over $x$.
\end{Definition}

Note that unlike in the case of canonical singularities of normal surfaces, $K_{\cF}$ need not be $\QQ$-Gorenstein at a canonical point of $(X, \cF)$ and cusps provide examples where $K_{\cF}$ is not $\QQ$-Gorenstein (see \cite[Theorem 1, III.3.2]{McQ08}).

If $h^0(mK_\cF) >0$ for some $m>0$, then we let $\phi _m:X\dasharrow \mathbb P ^N$ be the $m$-th pluricanonical map defined by the sections of $H^0(mK_\cF)$.
The {\it Kodaira dimension} of $\cF$ is defined by
\[  {\rm kod}(\cF)=\kappa (K_\cF)={\rm max}\{\dim \phi _m (X)|m\in \mathbb N \}\]
where by convention we let ${\rm kod}(\cF)=-\infty $ if $h^0(mK_\cF)=0$ for all $m>0$.
We say that $\cF$ is of general type if $ {\rm kod}(\cF)=\dim X$. It is well known that $\cF$ is of general type if and only if $K_\cF$ is big or equivalently ${\rm vol}(K_\cF)=\lim _{m\to \infty} h^0(mK_\cF) d!/m^d >0$, where $d=\dim X$.

\medskip

\begin{Definition} A  foliated surface $(X,\cF )$ is called a canonical model if $\cF$ is a foliation with canonical singularities on a complete normal surface $X$,  $K_\cF$ is nef and $K_\cF\cdot C=0$ implies $C^2\geq 0$ for any irreducible curve.\end{Definition}

\begin{Lemma}\label{l-can}
Let $(Y, \cG)$ be a canonical model. 
If ${\rm kod} (\cG)=2$ then $K_{\cG}$ is numerically ample, i.e., $K_{\cG}^2>0$ and $K_{\cG}\cdot C>0$ for every irreducible curve $C$ on $Y$.
\end{Lemma}

\begin{proof}
Let us assume that $K_{\cG}\cdot C=0$ for some irreducible curve $C$.
By the Hodge index theorem (see Lemma \ref{HIT}) we have $$K_{\cG}^2 \, C^2\le (K_{\cG}\cdot C)^2=0.$$
Since $K_{\cG}^2>0$ this implies $C^2\le 0$. Since $(Y, \cG)$  is a canonical model, this implies $C^2=0$.
Then again by the Hodge index theorem the class of $C$ is proportional to the class of $K_{\cG}$. Since $K_{\cG}^2>0$ this implies that $C$ is numerically trivial, which gives the required contradiction.
\end{proof}

Let $X$ be a complete normal surface $X$ and let $\cF$ be a foliation with canonical singularities.
If $K_{\cF}$ is pseudoeffective then by \cite[Theorem 1, III.3.2]{McQ08} there exists a proper birational morphism $f: X\to Y$ to a normal complete surface $Y$ such that $(Y, \cG=f_*\cF)$ is a canonical model.

\begin{Lemma}\label{crepant}
Let $f: (X, \cF)\to (Y, \cG)$ be a proper birational morphism of complete foliated surfaces with canonical foliation singularities. If  $K_{\cF}$ is  nef  then $K_{\cG}$ is nef and $K_{\cF}=f^*K_{\cG}$.
\end{Lemma}

\begin{proof}
Note that equality $K_{\cF}=f^*K_{\cG}$ and nefness of $K_{\cF}$ imply nefness of $K_\cG$, so it is sufficient to prove this equality.  Since $ (Y,\cG)$ has canonical singularities, we have $K_\cF =f^*K_{\cG}+E$ where $E$ is an effective and 
exceptional $\QQ$-divisor. Since $K_{\cF}$ is nef we have $K_{\cF}\cdot E\ge 0$. But $K_{\cF}\cdot E=E^2\le 0$, so $E^2=0$, which implies $E=0$. 
\end{proof}

\begin{Theorem}  Let $(X,\cF )$ be any foliation with at most canonical singularities. If $K_\cF$ is nef and big then there exists a unique morphism $f:(X,\cF )\to (X',\cF ')$ such that $K_\cF=f^*K_{\cF'}$ and $(X',\cF')$ is a canonical model.
\end{Theorem}

\begin{proof} The existence of $f:(X,\cF )\to (X',\cF ')$ is proved in \cite[Theorem 1 III.3.2]{McQ08}. By Lemma \ref{crepant} we have $K_\cF =f^*K_{\cF'}$.
Suppose that $g:(X,\cF )\to (X'',{\cF}''  )$ is also a morphism to a canonical model, then $f^*K_{\cF'}=K_\cF=g^*K_{{\cF }''}$ where both $K_{\cF'}$ and $ K_{{\cF }''}$ are numerically ample (see Lemma \ref{l-can}). Suppose that $C$ is a curve on $X$. Then $C\cdot K_\cF=f_*C\cdot K_{\cF'}=g_*C\cdot K_{\cF''}$ and so $C$ is contracted by $f$ if and only if $C\cdot K_\cF=0$, i.e. if and only if $C$ is contracted by $g$.  Since $X'$ and $X''$ are normal varieties, it follows that in fact $X'=X''$.
\end{proof}

\section{Contributions to the Riemann--Roch for canonical foliation singularities} \label{contributions}

In this section we use classification of canonical foliation singularities and we compute the corresponding 
contributions to the Riemann--Roch formula (see Subsection \ref{s-RR}). 
Unfortunately, the current classification as described in \cite{McQ08} describes only canonical singularities appearing 
on canonical models of complete foliated surfaces $(Y, \cG)$ with pseudoeffective $K_{\cG}$. In general, the classification is the same and can be done using McQuillan's ideas but the proof requires some additional work and we will deal with it in another paper. Since in this paper we study canonical models of general type,
we will use McQuillan's classification without further mentioning this fact.

Let  us also remark that the known classification provides only formal description in the $\QQ$-Gorenstein case (see \cite[Warning I.2.7]{McQ08}). However, $\QQ$-Gorenstein singularities of foliations occur only at quotient singularities. At such singularities $(Y, \cG, y)$ the local type of the reflexive sheaf
$\cO_Y (K_{\cG})$ at $y$ is determined by the formal description, so in these cases we will ignore the fact that the description is only formal.

\subsection{Terminal singularities}\label{terminal}

Let $(Y, \cG, y)$ be a terminal foliation singularity. Such a singularity is obtained by contracting an $\cF$-chain on a foliated surface $(X, \cF, C)$ such that $X$ is smooth and $\cF$ has only reduced singularities.  Let us recall that an $\cF$-chain is a Hirzebruch--Jung string $C=\bigcup C_i$ satisfying $K_{\cF}\cdot C_1=-1$ and $K_{\cF}\cdot C_i=0$ for $i>1$.  In particular,  the obtained singularity of $Y$ is cyclic of type  $\frac{1}{n}(1,q)$ for some pair of coprime integers $(n,q)$ with $0<q<n$ {(cf. \S \ref{cyclic}).  Therefore terminal foliation singularities are rational and $\QQ$-factorial.}

\begin{Lemma} \label{Brunella-cor}
$T_{\cG}$ is locally isomorphic at $y$ to the reflexive sheaf $\cL _{n-q}$. Moreover, we have  an isomorphism $T_{\cF}\simeq f^*T_{\cG}/ torsion$.
\end{Lemma}

\begin{proof}
By  Theorem \ref{Wunram} the line bundle  $\cM_q=f^*\cN_q/torsion$ has the same intersections with $C_i$
as $T_{\cF}$. To see this, recall that $s_0=n$ and $s_1=q$. Since $q=d_1s_1+t_1$ it follows that $d_1=1$ and $t_1=0$. It is then immediate that $t_i=d_i=0$ for $i=2,\ldots , r$. So Lemma \ref{equality} implies that $T_{\cF}$ and $f^*T_{\cG}/ torsion$ are isomorphic line bundles and
$T_{\cG}=(f_*T_{\cF})^{**}$ is locally  isomorphic to $f_*\cM_q=\cN_q=\cL _{n-q}$ on a neighborhood of $y$. 
\end{proof}
{\begin{Corollary}\label{C-ter} We have $\cO _Y(mK_\cG)\cong \cL _{\overline {mq}}$ where $\overline x$ denotes the reminder of dividing $x$ by $n$. In particular $a(y,K_\cG)=-\frac{n-1}{2n}$.
\end{Corollary}
\begin{proof} By Lemma  \label{Brunella-cor}, we have $\cO _Y(K_\cG )=\cL_q $ and hence $\cO _Y(mK_\cG)\cong \cL _{\overline {mq}}$. The claim now follows from \S \ref{terminal}. 
\end{proof}}
\medskip

\begin{Remark}
The above lemma is a slightly stronger version of  \cite[Corollary I.2.2]{McQ08}, which implies that the foliation 
$\cG$ is locally formally isomorphic to the quotient of a smooth foliation on $(\CC^2, 0)$ by the cyclic group $G$, whose generator acts by $(x_1,x_2)\to (\epsilon x_1, \epsilon^qx_2)$. If $v=\alpha \frac{\partial}{\partial x_1}+\beta \frac{\partial}{\partial x_2}$ is the vector field generating a smooth fibration on $(\CC^2, 0)$ then it transforms under the generator of $G$ to $v'=\alpha \epsilon^{-1}\frac{\partial}{\partial x_1}+\beta \epsilon^{-q}\frac{\partial}{\partial x_2}$ and the assertion that the corresponding foliation is invariant under $G$ is equivalent to the condition $v\wedge v'=0$. If $q=1$, any non-zero $v$ as above is $G$-invariant and it gives rise to a sheaf of type $\cL_1$. However, if $q\ne 1$ then we get the condition $\alpha\beta=0$. So the corresponding foliation corresponds to either $\frac{\partial}{\partial x_1}$ or $\frac{\partial}{\partial x_2}$. In the first case the tangent sheaf of the foliation is locally isomorphic to  $\cL _{n-1}$. By the above lemma this case does not occur if we have an $\cF$-chain. Thus we are in the second case and $\cG$ corresponds to  $\frac{\partial}{\partial x_2}$. 
\end{Remark}

\subsection{Canonical non-terminal $\QQ$-Gorenstein singularities} \label{can1}

\begin{Proposition}\label{Q-Gorenstein}
Let $(Y, \cG, y)$ be a canonical foliation singularity, which is $\QQ$-Gorenstein but it is not terminal.
Then one of the following holds:
\begin{enumerate}
\item $\cG$ is Gorenstein and $a(y, mK_{\cG})=0$ for all $m$, or
\item $\cG$ is $2$-Gorenstein and 
$$a(y, mK_{\cG})=\left\{
\begin{array}{cl}
0& \hbox{ for  $m$ even},\\
-\frac{1}{2} &\hbox{ for $m$ odd}.\\
\end{array}
\right.$$ 
\end{enumerate}
\end{Proposition}

\begin{proof}
By \cite[Fact I.2.4]{McQ08} we know that $(Y,y)$ has either a cyclic quotient singularity or a dihedral quotient singularity. The first case corresponds to cases (a)--(d) in  \cite[Fact I.2.4]{McQ08}  and in these cases $\cG$ is Gorenstein so we are in case 1 of the proposition. In the second case the assertion follows from Lemma \ref{dihedral}.
\end{proof}

\begin{Lemma}\label{dihedral}
Let $(Y, \cG, y)$ be a canonical foliation singularity and assume that $Y$ has a dihedral quotient singularity at $y$. Then for any integer $m$ we have
$$a(y, mK_{\cG})=\left\{
\begin{array}{cl}
0& \hbox{ for  $m$ even},\\
-\frac{1}{2} &\hbox{ for $m$ odd}.\\
\end{array}
\right.$$ 
\end{Lemma}

\begin{proof}
In the notation of \cite[Fact I.2.4]{McQ08}  let us consider a dihedral quotient singularity $(Y,y)=(\CC^2, 0)/G$, where $G\subset \GL (2, \CC)$ is a certain dihedral type group of order $4n$, that does not contain any pseudoreflections.
It is sufficient to show that $2K_{\cG}$ is Cartier and 
$a(y, K_{\cG})=-\frac{1}{2}.$

Let us consider case (e'). In this case $G\subset \GL (2, \CC)$ is generated by
$$\alpha=\left(  \begin{array}{cc}
\epsilon _{2n}  &0\\
0&\epsilon _{2n}^p\\
\end{array}
\right), \quad
\sigma=\left(  \begin{array}{cc}
0& i\\
i& 0\\
\end{array}
\right),
$$
where $p$ is a certain integer such that  $p\equiv -1 \mod 2^am$ and $p\equiv 1 \mod l$, the integers $l$, $m$ are odd and relatively prime  and  $2n=2^alm$.  In this case the foliation $\cG$  comes from a G-invariant foliation on $\CC^2$  generated by the vector field
$$\partial = (1 +\varphi ((xy)^l))x \frac{\partial} {\partial x} -  (1 +\varphi (-(xy)^l)) y\frac{\partial} {\partial y},$$
where $\varphi$ vanishes at $0$. Then one can easily check that $G$ acts on this vector field via
$$\partial  ^{\alpha}=\partial \quad \hbox{and} \quad \partial  ^{\sigma}=-\partial. $$ 
This shows that $T_{\cG}$ is the rank $1$ reflexive sheaf on $(Y,y)$ associated to the order $2$ 
character $\chi: G\to \CC^*$ defined by
$$\chi (\alpha) =1\quad \hbox{and} \quad \chi (\sigma)=-1.$$
In particular, the foliation $\cG$ is $2$-Gorenstein but it is not Gorenstein.

By \cite[Theorem 5.4]{La00} we have
$$a(y, \cO _Y (K_{\cG}))=\frac{1}{|G|}\sum _{g\in G : \, \chi (g)=-1} \frac{-2}{\det (1-g)}=- \frac {1}{2n}\sum _{j=0}^{2n-1} \frac {1}{\det (1-\sigma \alpha ^j)} = -\frac {1}{2n}\sum _{j=0}^{2n-1} \frac {1}{1+\epsilon _{2n}^{(p+1)j}} .$$

\begin{Lemma}
	For any positive integers $q$, $r$ we have
	$$  \sum _{j=0}^{q-1} \frac {1}{1+\epsilon _{q}^{rj}}
=\frac{q}{2}.$$
\end{Lemma}

\begin{proof}
	We have
	$$2  \sum _{j=0}^{q-1} \frac {1}{1+\epsilon _{q}^{rj}}
	=\sum _{j=0}^{q-1} \left(
	\frac {1}{1+\epsilon _{q}^{rj}}+\frac {1}{1+\epsilon _{q}^{r(q-j)}}\right)=\sum _{j=0}^{q-1} 
	\left( \frac {1}{1+\epsilon _{q}^{rj}}+\frac {\epsilon _{q}^{rj}}{1+\epsilon _{q}^{rj}}
	\right)=q.$$
\end{proof}

Applying the above lemma to $q=2n$ and $r=p+1$ we get
$$a(y, \cO _Y (K_{\cG}))=-\frac{1}{2}.$$

\bigskip

Now  let us consider the case (e''). In this case $G\subset \GL (2, \CC)$ is generated by
$$\alpha=\left(  \begin{array}{cc}
\epsilon _{2n}  &0\\
0&\epsilon _{2n}^p\\
\end{array}
\right), \quad 
\sigma=\left(  \begin{array}{cc}
0& \epsilon _{4n}^{ml}\\
\epsilon _{4n}^{ml} & 0\\
\end{array}
\right)=\left(  \begin{array}{cc}
0& \epsilon _{2^{a+1}}\\
\epsilon _{2^{a+1}} & 0\\
\end{array}
\right),
$$
where $p$ is a certain integer such that  $p\equiv 1 \mod 2^a$, $a\ge 2$, $p\equiv 1 \mod l$ and $p\equiv -1 \mod m$,  the integers $l$, $m$ are odd and relatively prime  and  $2n=2^alm$.

 In this case the foliation $\cG$  comes from a G-invariant foliation on $\CC^2$  generated by the vector field
$$\partial = (1 +\varphi ((xy)^{2^{a-1}l}))x \frac{\partial} {\partial x} -  (1 +\varphi (-(xy)^{2^{a-1}l})) y\frac{\partial} {\partial y},$$
where $\varphi$ vanishes at $0$. Then one can easily check that
$$\partial  ^{\alpha}=\partial \quad \hbox{and} \quad \partial  ^{\sigma}= -\partial. $$ 
So, as above, $T_{\cG}$ is the rank $1$ reflexive sheaf on $(Y,y)$ associated to the 
character $\chi: G\to \CC^*$ defined by
$$\chi (\alpha) =1\quad \hbox{and} \quad \chi (\sigma)=-1.$$
Similarly to the previous case we have
$$a(y, \cO _Y (K_{\cG}))= - \frac {1}{2n}\sum _{j=0}^{2n-1} \frac {1}{\det (1-\sigma \alpha ^j)} = -\frac {1}{2n}\sum _{j=0}^{2n-1} \frac {1}{  1-\epsilon _{2n}^{(p+1)j+ml}    } .$$

\begin{Lemma}
	For any positive integers $q$, $r$ and $s$ we have
	$$\Re \left( \sum _{j=0}^{q-1} \frac {1}{1-\epsilon _{q}^{rj+s}}
	\right)=\frac{q}{2}.$$
\end{Lemma}

\begin{proof}
	We have
	$$2\Re \left( \sum _{j=0}^{q-1} \frac {1}{1-\epsilon _{q}^{rj+s}}
	\right)=\sum _{j=0}^{q-1} \left(
	\frac {1}{1-\epsilon _{q}^{rj+s}}+\frac {1}{1-\epsilon _{q}^{-(rj+s)}}\right)=\sum _{j=0}^{q-1} 
	\left( \frac {1}{1-\epsilon _{q}^{rj+s}}-\frac {\epsilon _{q}^{rj+s}}{1-\epsilon _{q}^{rj+s}}
	\right)=q.$$
\end{proof}

Since $a(y, \cO _Y (K_{\cG}))$ is real,  the above lemma implies that
 $$ \sum _{j=0}^{2n-1} \frac {1}{  1-\epsilon _{2n}^{(p+1)j+ml} } =n
$$
and we get 
$$a(y, \cO _Y (K_{\cG}))=-\frac{1}{2}.$$
\end{proof}

\subsection{Canonical non-$\QQ$-Gorenstein singularities}\label{can2}

\begin{Proposition}\label{cusps}
Let $(Y, \cG, y)$ be a canonical foliation singularity and assume that $\cG$ 
is non-$\QQ$-Gorenstein at $y$. Then for any integer $m$ we have
$$a(y, mK_{\cG})=\left\{
\begin{array}{cl}
0& \hbox{ for }m=0,\\
-1&\hbox{ for }m\ne 0 .\\
\end{array}
\right.$$ 
\end{Proposition}

\begin{proof}
By \cite[Theorem III.3.2]{McQ08} $Y$ has a cusp singularity with $K_{\cF}\cdot C_i=0$ for all $i$. Here $f: (X, C)\to (Y, y)$ is the minimal resolution and $\cF =f^*\cG$.
Since $K_{\cG}=f_*K_{\cF}$ as divisors, this implies $K_{\cF}=f^*K_{\cG}$.
By \cite[Theorem IV.2.2]{McQ08} the divisor $K_{\cG}$ is not $\QQ$-Cartier, so the assertion follows by applying  Lemma \ref{RR-cusp} to $mK_{\cG}$.
\end{proof}

\medskip

Let $f: (X, C)\to (Y, y)$ be the minimal resolution of a cusp singularity. Let us set $Z=\sum _{i=1}^r C_i$, the sum of all exceptional curves. Then $K_X=f^*K_Y-Z=-Z$ as $Y$ is Gorenstein at $y$.

\begin{Lemma}\label{RR-cusp}
Let $D$ be a Weil divisor on $Y$ such that $f^*D\cdot C_i=0$ for all $i$. Then
$$a(y, D)=
\left\{
\begin{array}{cl}
0& \hbox{ if  $D$ is Cartier, }  \\
-1&\hbox{ if $D$ is not Cartier. } \\
\end{array}
\right.
$$
\end{Lemma}

\begin{proof}
By assumption we have $c_1(y, f^*D)=0$. Moreover, $f_*\cO_X (f^*D)=\cO_Y (D)$ by Sakai's projection formula (see \cite[Theorem 2.1]{Sa84}).
So  by formula (\ref{contribution-formula}) we have
$$a(y, D)= \chi (y, \cO_X (f^*D))- \chi (y, \cO_X)= \dim R^1f_*\cO_X (f^*D)- 1.$$
If $D$ is Cartier then $\cO_X (f^*D)\simeq \cO_X$
and the assertion is clear. So in the following we assume that $D$ is not Cartier.

Let us note that $R^1f_*\cO_X (f^*D-Z)= R^1f_*\cO_X (K_X+f^*D)=0$ (see, e.g., \cite[Theorem 2.2]{Sa84}).
Using the short exact sequence
$$0\to \cO_X (f^*D-Z)\to \cO_X (f^*D)\to \cO_Z (f^*D)\to 0,$$
we get an isomorphism $R^1f_*\cO_X (f^*D)\simeq H^1(Z,  \cO_Z (f^*D))$.
Note that $H^1(Z,  \cO_Z (f^*D))$ is dual to $H^0(Z,  \cO_Z (K_Z-f^*D))=H^0(\cO_Z (-f^*D))$.
By assumption $D$ is not Cartier, so by \cite[Theorem 4.2]{Sa84} $\cO_Z (f^*D)\not\simeq \cO_Z$.

Assume that  $H^0(\cO_Z (-f^*D))\ne 0$.  Then we have a nontrivial map $\varphi: \cO_Z\to \cO_Z (-f^*D)$.
Let $F$ denote its image. If the support of $F$ is equal to $C$, then the kernel of $\varphi$ is trivial. But since the Hilbert polynomials of $\cO_Z$  and $ \cO_Z (-f^*D)$ are the same (with respect to any ample polarization), the cokernel has trivial Hilbert polynomial. But the cokernel of $\varphi$ is a torsion sheaf, so it must be $0$ and $\varphi$ is an isomorphism, a contradiction.
This proves that there exists an exceptional curve $C_i$ not contained in
the support of $F$ but intersecting it non-trivially. But then $\varphi|_{C_i}:  \cO_{C_i}\to \cO_{C_i}(-f^*D)$
factors through the torsion sheaf $F_{C_i}$, so it is the zero map. Now let us  restrict $\varphi$ to the curve $C_{j}$ intersecting $C_i$ and contained in the support of $F$. Then  $\varphi|_{C_j}:  \cO_{C_j}\to \cO_{C_j}(-f^*D)\simeq \cO_{C_i}$ vanishes at the point $C_i\cap C_j$, so it must also be the zero map. But  $\varphi|_{C_j}$ 
is a composition of the surjection $ \cO_{C_j}\to F_{C_j}$ and a generic injection 
$F_{C_j}\to \cO_{C_j}(-f^*D)$, a contradiction.

This implies that  $H^0(\cO_Z (-f^*D))= 0$
and hence we have $R^1f_*\cO_X (f^*D)=0$, which implies the required equality.
\end{proof}

\section{Birational boundedness of weak nef models}

\begin{Definition}\label{d-wnm}
A  normal complete foliated surface $(Y, \cG)$ is called a weak nef model if the following conditions are satisfied:
\begin{enumerate}
\item $\cG$ has at most canonical singularities, 
\item  at singular points of $Y$ the foliation $\cG$ has only terminal singularities,
\item $K_{\cG}$ is nef.
\end{enumerate}
\end{Definition}

Let us note that by Theorem \ref{Artin} every weak nef model is projective.
By the proof of \cite[Proposition 5.1]{Br15} and by \cite[Theorem 8.1]{Br15} if $(X, \cF)$ is a smooth projective surface and $\cF$ has only reduced singularities then there exists a morphism $(X, \cF)\to (Y, \cG)$ to a weak nef model
(and such that $\cG$ has only reduced singularities on the smooth locus of $Y$). Therefore every birational equivalence class of foliations on normal surfaces contains at least one weak nef model. Let us remark that  birational equivalence classes of foliations tend to contain many weak nef models. This follows from the fact that a blow up of a weak nef model at a point where the surface is smooth but the foliation is singular, is still a weak nef model.

\medskip

\begin{Definition}
Let $Y$ be a normal complete surface. The \emph{index} $i(Y)$ of $Y$ is the smallest positive integer $m$ such that for every Weil divisor $D$ on $Y$ its multiple $mD$ is Cartier (if $Y$ is not $\QQ$-factorial, then we set $i(Y)=\infty$). The  \emph{index} $i(\cG)$ of a foliation $\cG$ on $Y$  is the smallest positive integer $m$ such that $mK_{\cG}$ is Cartier (if $K_{\cG}$ is not $\QQ$-Cartier, then we set $i(\cG)=\infty$). The \emph{$\QQ$-index} $i_{\QQ}(\cG)$ of a foliation $\cG$ on $Y$  is the smallest positive integer $m$ such that $mK_{\cG}$ is Cartier at all $\QQ$-Gorenstein points of the foliation.
\end{Definition}

\begin{Proposition}\label{sing-nef-model}
Let us fix a function $P: \ZZ_{\ge 0}\to \ZZ$. Then there exist some constants $B_1$, $B_2$ and $B_3$ (depending only on $P$) such that if $(Y, \cG)$ is a weak nef model with Hilbert function $\chi (Y, \cO_Y(mK_{\cG}))=P(m)$  for all $m\in \ZZ _{\ge 0}$ then 
$K_{\cG}^2=B_1$, $K_{\cG}\cdot K_Y=B_2$ and $\chi (Y,\cO_Y)=B_3$. Moreover, there exists some constants $C_1$ and $C_2$ such that the number of singularities of $Y$ is $\le C_1$ and the index of $Y$ is $\le C_2$.
\end{Proposition}

\begin{proof}
By Theorem \ref{RR} we have
$$P(m)=\chi (Y, \cO_Y(mK_{\cG})) =\frac{1}{2}mK_{\cG}(mK_{\cG}-K_Y)+\chi(Y, \cO_{Y})+\sum _{y\in \Sing Y}a(y, mK_{\cG}).$$ 
If $P$ is fixed then since $K_{\cG}^2$ is the quadratic term and $K_{\cG}\cdot K_Y$ is the linear term of $P$, both of them are fixed. Since $\chi(Y, \cO_{Y})=P(0)$, this number is also fixed. 

By assumption $\cG$ has only terminal singularities at singular points of $Y$.
So if $Y$ has a cyclic singularity   of type  $\frac{1}{n_y}(1,q_y)$  at $y\in \Sing Y$, then by Corollary \ref{C-ter}
 we have 
$$a(y, K_{\cG}) = - \frac {n_y-1}{2n_y} .$$
Now let us note that the number 
$$-\sum _{y\in \Sing Y}a(y,  K_{\cG})=-\chi (Y, \cO_Y(K_{\cG})) +\frac{1}{2}K_{\cG}(K_{\cG}-K_Y)+\chi(Y, \cO_{Y})$$
 is also fixed. We have
 $$-\sum _{y\in \Sing Y}a(y,  K_{\cG})=\sum _{y\in \Sing Y}\frac {n_y-1}{2n_y}\ge \frac{1}{4} |\Sing Y|, $$
 so the number of singularities of $Y$ is bounded. Now let us note that
 $$\sum _{y\in \Sing Y}\frac {1}{n_y}=  |\Sing Y|+ 2 \sum _{y\in \Sing Y}a(y, K_{\cG})$$
 assumes only a finite number of values. The proof of the proposition now follows from Lemma \ref{index-bound}.\end{proof}
 
 \begin{Lemma}\label{index-bound}
 Let us fix an integer $m$ and a rational number $c$. Then there exists only finitely many $m$-tuples $(n_1,...,n_m)$
 of positive integers $n_i$ such that $\sum _{i=1}^m \frac{1}{n_i}=c$.
 \end{Lemma}
 
 \begin{proof}
 The proof is by induction on $m$. For $m=1$ the assertion is trivial, so let us assume that it holds for all $(m-1)$.
 Without loss of generality we can assume that $n_1\le ...\le n_m$. Then $c\le \frac{m}{n_1}$, so $n_1\le \frac{m}{c}$.
 But then $n_1$ can assume only finitely many values and by the induction assumption 
 for each fixed $n_1$ the equation $\sum _{i=2}^m \frac{1}{n_i}=c-\frac{1}{n_1}$ has only finitely many solutions.
 \end{proof}

\begin{Lemma} \label{pseudo}
Let $D_1$ be a nef and big $\QQ$-divisor on a normal projective surface $Y$.
Let $D_2$ be another $\RR$-divisor such that $D_2+\alpha D_1$ is nef   for some $\alpha \ge 0$. Then either
$-D_2\equiv \alpha D_1$ is nef  or
the $\RR$-divisor $\beta D_1-D_2$ is pseudoeffective, where
$$\beta= \frac {  {2}\, D_1\cdot D_2 } {D_1^2} +\alpha .$$

\end{Lemma}

\begin{proof}
By pulling back the relevant divisors to an appropriate resolution of $Y$, we may assume that $Y$ is smooth.  By \cite[Theorem 2.2.15]{Laz04} we know that if 
$$D_1^2> 2 tD_1\cdot ( D_2+\alpha D_1)$$
then $D_1-t (D_2 +\alpha D_1)$ is big. In particular, since 
the limit of big divisors is pseudoeffective, we see that if 
$$D_1^2\ge 2 tD_1\cdot ( D_2+\alpha D_1) \leqno{(*)}$$
then $D_1-t (D_2 +\alpha D_1)$ is pseudoeffective.

If $D_1\cdot (D_2+\alpha D_1) = 0$ then by the Hodge index theorem 
$$(D_1\cdot (D_2+\alpha D_1) )^2= 0\ge D_1^2  (D_2+\alpha D_1)^2\ge 0.$$
Since $D_1^2>0$ this implies that $(D_2+\alpha D_1)$ is numerically trivial, so $-D_2$ is nef. 
In the following we can therefore assume that $D_1\cdot (D_2+\alpha D_1) > 0$.

In this case  let us set
$$t_0= \frac{D_1^2} { {2}\, D_1\cdot (D_2+\alpha D_1)} .$$
Inequality $(*)$  is satisfied for all $0\le t\le t_0$.
In particular, $D_1-t _0(D_2 +\alpha D_1)$ is pseudoeffective.
This implies that $\beta D_1-D_2=\frac{1-t_0\alpha}{t_0}D_1-D_2$ is also pseudoeffective.
\end{proof}

\begin{Theorem}\label{birational-boundedness-nef}
Let us fix a function $P:\ZZ_{\geq 0}\to \ZZ$ and consider the family of  weak nef models $(Y, \cG)$
such that $\cG$ is of general type and $\chi (Y, \cO _Y(mK_{\cG}))=P(m)$ for all $m\ge 0$. 
Then there exists a constant $N_1$ depending only on $P$ such that for all $(Y, \cG)$ in the above family,
the linear system $|mK_{\cG}|$ gives a birational map for all $m\ge N_1$.
\end{Theorem}

\begin{proof}
Let $f:Y\to Y'$ be the morphism to the canonical model of $(Y,\cG )$ and $\cG'$ the induced foliation on $Y'$.
By Lemma \ref{crepant} we have $K_\cG=g^*K_{\cG'}$. { By Proposition \ref{sing-nef-model}, the index $i(\cG)$ of $K_\cG$ is bounded. By Lemma \ref{l-can},  $K_{\cG '}$ is numerically ample and so
 \[ i(\cG)K_{\cG '}\cdot C=i(\cG)K_{\cG}\cdot f^{-1}_*C\geq 1\]
for any  curve $C\subset Y'$.  We claim that $K_{Y'}+{ 3}i(\cG)K_{\cG '}$ is nef.
To this end, let $\nu :Y''\to Y'$ be the morphism obtained by taking the minimal resolution of the cusps of $Y'$.
By \cite[Theorem III.3.2]{McQ08} (see also Case 4 in \S 5), $Y''$ has rational singularities and so by  Lemma \ref{Artin}, $Y''$ is projective. Note that over each cusp of $Y'$, the exceptional curve corresponds to a cycle of smooth rational curves or to a rational curve with one node.
We have $K_{Y''}+E=\nu ^* K_{Y'}$ and $K_{\cG''}=\nu ^*K_{\cG'}$.
By \cite[Proposition 3.8]{Fuj12}, every $K_{Y''}$ negative extremal ray is spanned by a rational curve $C$ with $0<-K_{Y''}\cdot C\leq 3$ and therefore $K_{Y''}+{ 3}i(\cG)K_{\cG ''}$ is nef (see \cite[Theorem 3.2]{Fuj12}). But then $K_{Y'}+3i(\cG)K_{\cG'}=\nu _*(K_{Y''}+{ 3}i(\cG)K_{\cG ''})$ is also nef. }

Note that $K_{\cG'}^2=K_{\cG}^2$ and $K_{\cG'}\cdot K_{Y'}=K_{\cG}\cdot K_Y$.

By Lemma \ref{pseudo} we know that if 
$$\gamma  =\max \left(  \frac {  {2}\, K_{\cG}\cdot K_Y  } {K_{\cG}^2} +3i(\cG) , 0 \right)$$
then $\gamma K_{\cG'}-K_{Y'}$ is pseudoeffective.  
Let us note that  $L=(4  i(\cG) +1+a) K_{\cG '}+(\lceil \gamma \rceil K_{\cG'}-K_{Y'})$ is pseudoeffective for any $a\ge 0$. Let $L=P+N$
be its Zariski decomposition.  Thus $$P^2> (4 i(\cG)K_{\cG'}) ^2=(4 i(\cG)K_{\cG}) ^2\ge 16.$$
and for any curve $C$ not contained in the negative part $N'$ of the Zariski decomposition of  $\gamma K_{\cG'}-K_{Y'}$  we have
$$P\cdot C\ge (4i(\cG) +1){ K_{\cG'} }\cdot C> 4.$$
So by Theorem \ref{Reider-type} the linear system $|K_{Y'}+L |= |( 4  i(\cG)+\lceil \gamma \rceil +1+a)K_{\cG'}|$ 
separates any two (possibly infinitely near) points lying in the smooth locus of $Y-\Supp N'$ and, in particular, the corresponding map is birational. Let us recall that by \cite[Theorem 6.2]{Sa84} we have $f_*\cO_Y(mK_{\cG}) =\cO_X(m K_{\cG'})$ for any positive $m$. 
Therefore  $|( 4  i(\cG)+\lceil \gamma \rceil +1+a)K_{\cG}|=  |( 4  i(\cG)+\lceil \gamma \rceil +1+a)K_{\cG'}|$ and hence this linear system also defines a birational map.
\end{proof}

\section{Birational boundedness of canonical models}

\begin{Proposition} \label{sing-can-model}
Let us fix a function $P: \ZZ_{\ge 0}\to \ZZ$. Then there exist some constants $B_1$, $B_2$, $B_3$ and $B_4$ such that if $(Y, \cG)$ is a canonical model with the Hilbert function $\chi (Y, \cO_Y(mK_{\cG}))=P(m)$  for all $m\in \ZZ _{\ge 0}$ then 
$K_{\cG}^2=B_1$, $K_{\cG}\cdot K_Y=B_2$, $\chi (Y,\cO_Y)=B_3$ and the number of cusps of $Y$ is equal to $B_4$. Moreover, there exists some constants $C_1$ and $C_2$ such that the number of terminal and dihedral singularities of $(Y, \cG)$ is $\le C_1$. Moreover, the index of the { surface $Y$} at any terminal foliation singularity is $ \le C_2$. In particular, $2 C_2 K_{\cG} $ is Cartier at all non-cusp singularities of $\cG$, so that $i_{\QQ}(\cG)\le 2C_2$.
\end{Proposition}
{Recall that terminal singularities are discussed in \S \ref{terminal} and dihedral singularities are canonical singularities of index 2 discussed in \S \ref{can1}.}
\begin{proof}
As in proof of Proposition \ref{sing-nef-model} the numbers $K_{\cG}^2$, $K_{\cG}\cdot K_Y$ and $\chi (Y,\cO_Y)$
can be determined from the Hilbert function $P$. It follows that  the number 
$\sum _{y\in \Sing Y}a(y,  K_{\cG})$ is fixed. Let $\Sigma_1$ be the set of singular points of $Y$ at which $(Y, \cG)$
is terminal.  Similarly, let $\Sigma_2$ be the set of  dihedral quotient singularities of $Y$  and $\Sigma_3$ the set of cusps of $Y$. Let us set $\Sigma=\Sigma_1\cup \Sigma_2\cup \Sigma _3$.
Then by the results of Section \ref{contributions} we have
 $$-\sum _{y\in \Sing Y}a(y,  K_{\cG})=\sum _{y\in \Sigma_1}\frac {n_y-1}{2n_y}
 +\sum _{y\in \Sigma_2}\frac{1}{2}+\sum _{y\in \Sigma_3}1  \ge \frac{1}{4} |\Sigma|.$$
 Therefore $|\Sigma|$ is bounded and hence
 $$\sum _{y\in \Sigma _1}\frac {1}{n_y}=  |\Sigma |+|\Sigma _3|+ 2 \sum _{y\in \Sing Y}a(y, K_{\cG})$$
 assumes only a finite number of values.   So by Lemma \ref{index-bound} the indices of the { surface $Y$} at terminal foliation singularities are bounded by some constant $C_2$ depending only on $P$. Then the last assertion follows from Proposition \ref{Q-Gorenstein}.
 
Finally, let us set $m= (2C_2)!$. Note that $m$ depends only on $P$ and it  is a multiple of $i_{\QQ} (\cG)$. { Let $\nu :Y'\to Y$ be the minimal resolution of the cusps on $Y$ so that $Y'$ is projective and let $\cG'=\nu ^*\cG$ so that $K_{\cG'}=\nu ^*K_{\cG}$. Recall that by the proof of  Proposition \ref{cusps}, we know that since $m>0$, $R^1\nu _* \cO _{Y'}(mK_{\cG '})=0$ and $\nu _* \cO _{Y'}(mK_{\cG'})\cong \cO _{Y}(mK_{\cG })$. We also have $\nu _* \cO_{Y'}=\cO _Y$ and ${\rm length}(R^1\nu _* \cO _{Y'})=|\Sigma _3|$. 
Thus, by the Leray spectral sequence,
$$P(m)=\chi (Y, \cO_Y(mK_{\cG})) =\chi (Y', \cO_{Y'}(mK_{\cG'}))=$$ 
$$\frac{1}{2}mK_{\cG'}(mK_{\cG'}-K_{Y'})+\chi(Y', \cO_{Y'})=\frac{1}{2}mK_{\cG}(mK_{\cG }-K_{Y})+\chi(Y, \cO_{Y})-|\Sigma_3|, $$}
so the number of cusps of $Y$ depends only on the Hilbert function $P$.
\end{proof}

\begin{Remark}
In case of canonical models we cannot bound the index of $Y$ for fixed Hilbert function $P$. The problem is that
introducing, e.g., canonical non-terminal singularities of the foliation at cyclic quotient singularities does not change
the Hilbert function of $mK_{\cG}$.
\end{Remark}

\begin{Theorem}\label{birational-boundedness-canonical}
Let us fix a function $P:\ZZ_{\geq 0}\to \ZZ$ and consider the family of canonical models of foliations $(Y, \cG)$
such that $\cG$ is of general type and $\chi (Y, \cO_Y (mK_{\cG}))=P(m)$ for all $m\ge 0$. 
Then there exists a constant $N_1$ depending only on $P$ such that for all $(Y, \cG)$ in the above family
and for all $m\ge N_1$, the linear system $|mK_{\cG}|$ defines a birational map.
\end{Theorem}

\begin{proof}
Let $g: (Z,\cH)\to (Y, \cG)$ be the minimal resolution of the cusps. {In particular there are no $-1$ curves over $Y$ and hence $K_Z$ is nef over $Y$. Since $Z$ has only rational singularities, by Lemma \ref{Artin}, $Z$ is projective.} We have $K_{\cH}=g^*K_{\cG}$
and by Proposition \ref{sing-can-model}, $i(\cH)=i_{\QQ}(\cG)$ where $i(\cH )\leq 2C_2$. 
For any curve  $C$ on $Y$ we have $K_{\cH}\cdot g_*^{-1} C=K_{\cG}\cdot C>0$. Therefore we have $i_{\QQ}(\cG) K_{\cH}\cdot  g_*^{-1}  C\ge 1$.
As in the proof of Theorem \ref{birational-boundedness-nef}  this implies that $K_Z+3i_{\QQ}(\cG) K_{\cH}$ is nef.

By Lemma \ref{pseudo} we know that if 
$$\gamma  =\max \left(  \frac {  {2}\, K_{\cG}\cdot K_Y } {K_{\cG}^2} + 3i_{\QQ}(\cG), 0 \right)$$
then
 $\gamma K_{\cH}-K_Z$ is pseudoeffective. Let us note that  $L=(4  i_{\QQ}(\cG) +1+a) K_{\cH}+(\lceil \gamma \rceil K_{\cH}-K_Z)$ is pseudoeffective for any $a\ge 0$. Let $L=P+N$
be the Zariski decomposition. Then we have 
$$P^2> (4 i_{\QQ}(\cG)K_{\cH}) ^2 \ge  16.$$
If $C$ is a curve not contained in the negative part of the Zariski decomposition of  $\gamma K_{\cH}-K_Z$  then
$$P\cdot C\ge (4i_{\QQ}(\cG) +1)K_{\cH} \cdot C> 4.$$
By Theorem \ref{Artin} $Z$ is projective so we can apply Theorem \ref{Reider-type}  to 
the linear system $|K_Z+L |= |( 4  i_{\QQ}(\cG)+\lceil \gamma \rceil +1+a)K_{\cH}|$.
As in the proof of Theorem \ref{birational-boundedness-nef} we conclude that it defines a birational map.
Moreover,  by \cite[Theorem 6.2]{Sa84} we have 
 $|( 4  i_{\QQ}(\cG)+\lceil \gamma \rceil +1+a)K_{\cH}|=  |( 4  i_{\QQ}(\cG)+\lceil \gamma \rceil +1+a)K_{\cG}|$, so this linear system also defines a birational map.
\end{proof}

{

\section{Partial crepant resolution of a canonical foliation singularity} \label {special-minimal}

Let $(Y, \cG,y )$ be a canonical foliation singularity. {As mentioned in \S \ref{contributions} we only consider canonical singularities arising on canonical models of foliated surfaces of general type (see Sections \ref{terminal}, \ref{can1}, \ref{can2} as well as \cite[Corollary I.2.2, Fact I.2.4 and Theorem III.3.2]{McQ08} for a description of terminal and canonical singularities)}. Let $f':(Y', C')\to (Y, y)$ be the minimal resolution of $(Y, y)$ and $\cG'=(f')^*\cG$. Let us consider the normal surface $(X, C)$ obtained by contracting all maximal $\cG'$-chains contained in $C'$ {(see \S \ref{terminal} for the definition of $\cG'$-chains)}.  We say that $(X,C)\to (Y,y)$ is the {\it minimal partial crepant resolution}. Note that as we will see below, $X$ may have singularities of type $\frac 12(1,1)$ which are contained in $C$ (see cases (2) and (3) below).
 \cite[Theorem III.3.2]{McQ08} implies that we get an induced morphism $f: (X, C)\to (Y, y)$ such that 
$K_{f^*\cG}=f^*K_{\cG}$. Let $\cF =f^*\cG$.

\begin{Theorem} \label{strange-vanishing}
With the notation above, we have $f_*\cO_X(mK_{\cF})= \cO_Y(mK_{\cG})$ for all $m$ 
and $R^1 f_* \cO _X (mK_{\cF})=0$ for $m\ne 0.$
In particular,  we have
$$\chi (y, \cO _X(mK_{\cF}) )= \left\{ 
\begin{array}{cl}
1 & \hbox{ if $m=0$ and $y$ is a cusp,}\\
0 & \hbox{ otherwise.}\\
\end{array}
\right.
$$
\end{Theorem}

\begin{proof}
{Note that if $y\in Y$ is terminal, then $f$ is an isomorphism and there is nothing to prove.} If $f$ is not an isomorphism at $y\in Y$ then by \cite[Theorem III.3.2]{McQ08} we have
four possibilities for the exceptional curve $C=f^{-1}(y)$ { (we caution the reader that our enumeration of these cases is presented differently than that of \cite[Theorem III.3.2]{McQ08})}. In all theses cases the irreducible components of $C$ are rational curves $C_i$ satisfying $K_{\cF}\cdot C_i=0$.

\medskip

\subsection{Case 1}

In this case $Y$ has a cyclic quotient singularity at $y$ and $\cG$ is Gorenstein at $y$.
The curve $C$ consists of a chain of smooth rational curves.
\medskip

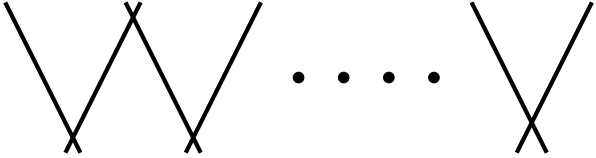
\begin{tikzpicture}
\draw[black, ultra thick] (-2.0,1) -- (-1,-1);
\draw[black, ultra thick] (-1.2,-1) -- (-.2,1);
\draw[black, ultra thick] (-.4,1) -- (.6,-1);
\draw[black, ultra thick] (.4,-1) -- (1.4,1);
\filldraw [black] (1.9,0) circle (2pt);
\filldraw [black] (2.5,0) circle (2pt);
\filldraw [black] (3.1,0) circle (2pt);
\filldraw [black] (3.7,0) circle (2pt);
\draw[black, ultra thick] (4.2,1) -- (5.2,-1);
\draw[black, ultra thick] (4.8,-1) -- (5.8,1) ;

\end{tikzpicture}

\medskip

Since $\cG$ is Gorenstein we have
$$\chi (y, \cO _X (mK_{\cF}))= 0$$
for all integers $m$.

\subsection{Case 2}

In this case $Y$ has a cyclic quotient singularity at $y$ and $\cG$ is Gorenstein at $y$. 
The curve $C$ is smooth rational and it passes through $2$ singular points of type $\frac{1}{2}(1,1)$. By Lemma \ref{Brunella-cor},
$\cO_{X}(K_{\cF})$ is of type $\cL_1$ at both these points.

\medskip

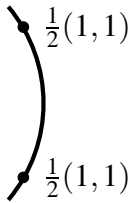
\begin{tikzpicture}
\draw[black, ultra thick, -] (0,0) arc (-40:40:2);
\filldraw [black] (0.2,.3) circle (2pt)node{\qquad\qquad $\frac 1 2 (1,1)$};
\filldraw [black] (0.2,2.3) circle (2pt)node{\qquad\qquad $\frac 1 2 (1,1)$};
\end{tikzpicture}

\medskip

As in the first case we have
$$\chi (y, \cO _X (mK_{\cF}))= 0$$
for all integers $m$.

\subsection{Case 3} 
 
In this case $Y$ has a dihedral quotient singularity at $y$ and $\cG$ is $2$-Gorenstein at $y$. 
The curve $C$ is consists of a chain of smooth rational curves, in which the first curve passes through 
$2$ singular points of type $\frac{1}{2}(1,1)$. As in the previous case
$\cO_{X}(K_{\cF})$ is of type $\cL_1$ at both these points.

\medskip

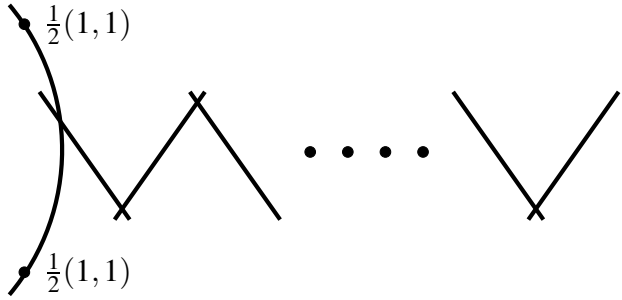
\begin{tikzpicture}
\draw[black, ultra thick, -] (0,0) arc (-40:40:3);
\filldraw [black] (0.2,.3) circle (2pt)node{\qquad\qquad $\frac 1 2 (1,1)$};
\filldraw [black] (0.2,3.6) circle (2pt)node{\qquad\qquad $\frac 1 2 (1,1)$};
\draw[black, ultra thick] (0.4,2.7) -- (1.6,1.0);
\draw[black, ultra thick] (1.4,1.0) -- (2.6,2.7);
\draw[black, ultra thick] (2.4,2.7) -- (3.6,1.0);
\filldraw [black] (4.0,1.9) circle (2pt);
\filldraw [black] (4.5,1.9) circle (2pt);
\filldraw [black] (5.0,1.9) circle (2pt);
\filldraw [black] (5.5,1.9) circle (2pt);
\draw[black, ultra thick] (5.9,2.7) -- (7.1,1.0);
\draw[black, ultra thick] (6.9,1.0) -- (8.1,2.7) ;

\end{tikzpicture}

\medskip

Since $2K_{\cG}$ is Cartier we have
$$\chi (y, \cO _X (mK_{\cF}))= \left\{
\begin{array}{cl}
0&\hbox{ if $m$ is even,}\\  
\chi (y, \cO _X(K_{\cF})) &\hbox{ if $m$ is odd.}\\
\end{array}
\right.$$

Let $g: Z\to X$ be the blow up at two singular points $x_1$, $x_2$ lying on $C$.
By the results of \S \ref{mod-Euler} we have
$$ \chi (y, \cO_X (K_{\cF}))=  \chi (y, \cO_Z (K_{g^*\cF})) -\chi (x_1,  \cO_Z (K_{g^*\cF}))-\chi (x_2,  \cO_Z (K_{g^*\cF})).$$
Using the definition of  the Riemann--Roch contributions (see \eqref{contribution-formula} in \S \ref{s-RR}) we get
$$ \chi (y, \cO_Z (K_{g^*\cF})) = a(y, K_{\cG})-\frac{1}{2}c_1(y, K_{g^*\cF}) (c_1(y, K_{g^*\cF}) -K_Z)$$
and
$$ \chi (x_i, \cO_Z (K_{g^*\cF})) = a(x_i, K_{\cF})-\frac{1}{2}c_1(x_i, K_{g^*\cF}) (c_1(x_i, K_{g^*\cF}) -K_Z)$$
for $i=1, \, 2$. Since $K_{\cF}=f^*K_{\cG}$ we have
$$c_1(y, K_{g^*\cF}) =c_1(x_1, K_{g^*\cF})+c_1(x_2, K_{g^*\cF})$$
as can be easily seen by intersecting both sides with all irreducible components of the exceptional divisor of $f\circ g$.
So using  the results of Section \ref{contributions}  we have
$$  \chi (y, \cO_X (K_{\cF}))= a(y, K_{\cG})-a(x_1, K_{\cF})-a(x_2, K_{\cF}) =-\frac{1}{2}+\frac{1}{4}+\frac{1}{4}=0.$$
This shows that
$$\chi (y, \cO _X (mK_{\cF}))= 0$$
for all integers $m$.

\subsection{Case 4}

In this case $Y$ has a cusp at $y$ and $\cG$ is not $\QQ$-Gorenstein at $y$.
The curve $C$ is either a cycle of smooth rational curves or a rational curve with one node.

\medskip

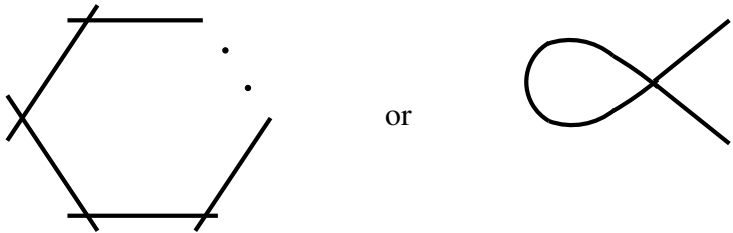
\begin{tikzpicture}
\draw[black, ultra thick] (-.2,-.3) -- (1.0,1.5);
\draw[black, ultra thick] (-.2,.3) -- (1.0,-1.5);
\draw[black, ultra thick] (.6,1.3) -- (2.4,1.3);
\draw[black, ultra thick] (.6,-1.3) -- (2.6,-1.3);
\draw[black, ultra thick] (2.3,-1.5) -- (3.3,.0);
\filldraw [black] (2.7,.9) circle (1pt);
\filldraw [black] (3.0,.4) circle (1pt);
\filldraw [black] (5.0,.0) circle (0pt)node{or};
\draw[black, ultra thick, -] (7,1) arc (120:240:.6);
\draw[black, ultra thick, -] (7.85,.83) arc (50:110:.89);
\draw[black, ultra thick, -] (7.0,-.04) arc (250:310:.89);
\draw[black, ultra thick, -] (8.45,.41) arc (50:60:4.2);
\draw[black, ultra thick, -] (7.85,.09) arc (300:310:4.2);
\draw[black, ultra thick] (8.4,.49) -- (9.4,1.3);
\draw[black, ultra thick] (8.4,.46) -- (9.4,-.34);

\end{tikzpicture}

\medskip

In this case  we have
$$\chi (y, \cO _X(mK_{\cF}))= \left\{ 
\begin{array}{cl}
1& \hbox{ if $m=0$,}\\
0& \hbox{ if $m\ne 0$.}\\
\end{array}
\right.$$
This follows immediately from formula (\ref{contribution-formula}) in \S \ref{s-RR} and Lemma \ref{RR-cusp} (or from the proof
of  Lemma \ref{RR-cusp}).

\medskip

Since $K_{\cF}=f^* K_{\cG}$,  the equality $f_*\cO_X(mK_{\cF})= \cO_Y(mK_{\cG})$ follows from
\cite[Theorem 6.2]{Sa84}. So for $m\ne 0$ we have
$$\dim R^1 f_* \cO _X (mK_{\cF}) = \chi (y, \cO _X (mK_{\cF}))= 0.$$
\end{proof}

\medskip

\section{Vanishing theorems for foliations}

The main aim of this section is to prove the following theorems.

\begin{Theorem}\label{strong-vanishing}
Let $(X, \cF)$ and $(Y, \cG)$ be foliated surfaces with only canonical singularities. If 
$f: (X, \cF)\to (Y, \cG)$ is a proper birational morphism then $f_* \cO_X(mK_{\cF})=\cO_Y(mK_{\cG})$ for any non-negative integer $m$ and we have
$$R^i f_* \cO_X(K_{\cF})=0$$
for $i>0$.
\end{Theorem}

\begin{Theorem}\label{crepant-vanishing}
Let $(X, \cF)$ and $(Y, \cG)$ be foliated surfaces with only canonical singularities. If 
$f: (X, \cF)\to (Y, \cG)$ is a proper birational morphism and $K_{\cF}=f^*K_{\cG}$  then 
$$R^1f_* \cO_X(mK_{\cF})=0$$ for any $m\ne 0$. 
\end{Theorem}

The main ingredients in the proofs of the above theorems are Theorem \ref{strange-vanishing} and the following lemmas.

\begin{Lemma}\label{smooth-vanishing}
Let $ (Y, y)$ be a germ of a smooth surface and let $\cG$ be a foliation with canonical singularity at $y$.
Let $f:(X, C)\to (Y,y)$ be the blow up at $y$ and set $\cF=f^*\cG$. Then $R^1f_*\cO_X (K_{\cF})=0$.
\end{Lemma}

\begin{proof}
Let $C$ denote the exceptional divisor of $f$.  By definition of canonical singularities we have $K_\cF-f^*K_{\cG}= m(y) C$ for some $m(y)\ge 0$. Note that we have the equality
$$T_{\cF}=f^*T_{\cG} \otimes \cO_X((l(y)- 1)C),$$
where $l(y)\ge 0$ is the vanishing order of the form $f^*\omega$ along $C$ and $\omega$ is the $1$-form 
defining $\cG$ (see \cite[Chapter 2, Section 3, (1)]{Br15}). Therefore $m(y)=1-l(y)\leq 1$ and so
$m(y)=0$ or $m(y)=1$. 
If $m(y)=0$, then $f^*\cO _Y(K_{\cG})\to \cO_X (K_{\cF})$ is an isomorphism. By the projection formula, we have $R^1f_*f^*\cO _Y(K_{\cG})= R^1f_*\cO_X\otimes \cO_Y(K_{\cG})=0$ and so the assertion is clear.  If $m(y)=1$ consider  the short exact sequence
$$0\to f^*\cO _Y(K_{\cG})\to \cO_X (K_{\cF})\to \cO_C(C)\to 0.$$

Pushing this forward, we obtain the exact sequence
$$0= R^1f_*f^*\cO _Y(K_{\cG})\to R^1f_*\cO_X (K_{\cF})\to R^1f_* \cO_C(C)\to 0.$$
Since $ R^1f_* \cO_C(C)= H^1 (\PP^1, \cO _{\PP ^1}(-1)) =0$, then $ R^1f_*\cO_X (K_{\cF})=0$.
\end{proof}

\begin{Lemma}\label{cyclic-vanishing}
Let $(X, \cF, C)\to (Y, \cG, y)$ be a contraction of an $\cF$-chain to a singularity of type $\frac{1}{n}(1,q)$ as in Subsection \ref{terminal}. Then
$$\chi (y, \cO _X(mK_{\cF} ))=\frac{1}{n} \left(  \frac{(m-\overline{mq})  (n-1)}{2} +\frac{m(m-1)q}{2} + \sum _{j=0}^{\overline{mq}-1}\overline{cj}  \right) ,$$
where $c$ denotes an integer such that $qc \equiv -1\mod n$ and $\overline x$ 
denotes the remainder from dividing $x$ by $n$. 
\end{Lemma}

\begin{proof}
For the exceptional curve $C=\bigcup C_i$ we set $C_i^2=-b_i$ for some $b_i\ge 2$.
Let us write $c_1(y, K_{X})= \sum x_i C_i$ and  $c_1(y, K_{\cF})= \sum y_i C_i$. The rational numbers $x_i$ and $y_i$ are uniquely determined by the following systems of linear equations:
$$( \sum x_i C_i) \cdot C_j= K_X\cdot C_j=-C_j^2-2=b_j-2
$$
and
$$( \sum y_i C_i) \cdot C_j= \left\{
\begin{array}{cl}
-1& \hbox{ for $j=1$,}\\
0& \hbox{ for $j>1$.}
\end{array}
\right.$$
Solving these systems of equations one can easily see that
$$x_1=-1+\frac{q+1}{n}\quad \hbox{ and }\quad y_1=\frac{q}{n}.$$
Therefore we have
$$\begin{aligned}
\chi (y, \cO _X(mK_{\cF} ))=&a (y, mK_{\cG})-\frac{1}{2}mc_1(y, K_{\cF})\left(
mc_1(y, K_{\cF})-c_1(y, K_{X})\right)\\
=&a (y, mK_{\cG})- \frac{1}{2}m (\sum y_i C_i) (  \sum (my_j-x_j) C_j)\\
=&a (y, mK_{\cG})-\frac{1}{2}m (my_1-x_1).
\end{aligned}$$
Since by Lemma \ref{Brunella-cor} the sheaf $\cO _X(mK_{\cG} )$ is locally of type $\cL_{\overline{mq}}$, the required formula follows from the above and  the corresponding formula for $a (y, \cL_{\overline{mq}})$ 
from \S \ref{s-RR}.
\end{proof}

\medskip

\begin{Remark}
Let us note that the formula in Lemma \ref{cyclic-vanishing} gives vanishing of $\chi (y, \cO _X(K_{\cF} ))$. However, unlike in Theorem \ref{strange-vanishing}, $\chi (y, \cO _X(mK_{\cF} ))$  is usually non-zero for $m\ge 2$. For example, for a terminal foliation on the singularity of type $\frac{1}{3}(1,1)$ we have $\chi (y, \cO _X(2K_{\cF} ))=1$. 
In fact, the vanishing of $\chi (y, \cO _X(mK_{\cF} ))$ fails for $m\ge 2$ already in the situation of Lemma \ref{smooth-vanishing} (if $\cG$ is regular at $y$).
\end{Remark}

\medskip

\begin{proof}[Proof of Theorem \ref{strong-vanishing}.]
The equality  $f_* \cO_X(mK_{\cF})=\cO_X(mK_{\cG})$ for $m\ge 0$ follows from the definition of canonical singularities and \cite[Theorem 6.2]{Sa84}. Hence vanishing of $R^i f_* \cO_X(K_{\cF})$  is equivalent to vanishing of $\chi (y, \cO _X(K_{\cF} ))=\dim R^1 f_* \cO_X(K_{\cF})_y$ for all points $y\in Y$.

Let $g: (Z, \cH)\to (X, \cF)$ be a proper birational morphism such that $f\circ g$ dominates the minimal resolution of singularities of $Y$. By \S \ref{mod-Euler} we have
$$\chi (y, \cO_Z ( K_{\cH}))= \chi (y,  \cO_Y (K_{\cF}))+\sum _{x\in f^{-1}(y)} \chi (x, \cO_Z (K_{\cH})) .$$
Since all the numbers are non-negative, it is sufficient to prove that $\chi (y, \cO_Z ( K_{\cH}))=0$.

Now let us remark that by assumption $f\circ g$ factors into a composition of maps considered in Theorem  \ref{strange-vanishing} and Lemmas \ref{smooth-vanishing} and \ref{cyclic-vanishing}. Since for each of these maps we have vanishing of the modified Euler characteristics $\chi (\tilde y, \cO (K_{\tilde \cF} ))$, we have also vanishing of $\chi (y, \cO_Z ( K_{\cH}))$. This finishes the proof of Theorem \ref{strong-vanishing}.\end{proof}

\medskip

\begin{Remark}
Since we have a generically surjective map $\Omega _{X}\to \cO_X (K_{\cF})$, the vanishing of  
$R^1 f_* \cO_X(K_{\cF})$ would follow from the vanishing of $R^1f_*\Omega_X$. Unfortunately, this last 
group is usually non-zero. For example if $f$ is the minimal resolution of a quotient singularity then one can show that 
the dimension of the stalk of $R^1f_*\Omega_X$ at the singularity is equal to the number of exceptional curves in the resolution.
\end{Remark}

\medskip

\begin{Corollary}
Let $(X, \cF)$ and $(Y, \cG)$ be  foliated complete surfaces with only canonical singularities. If 
$(X, \cF)$ and $(Y, \cG)$ are birationally equivalent then
$$h^0(X, \cO_X(mK_{\cF}))=h^0(Y, \cO_Y(mK_{\cG}))$$
for all $m\ge 0$, and  
$$h^i(X, \cO_X(K_{\cF}))=h^i(Y, \cO_Y(K_{\cG}))$$
for all $i$. In particular, we have $\chi (X, \cO_X(K_{\cF}))=\chi (Y, \cO_Y(K_{\cG})).$
\end{Corollary}

\begin{proof}
There exists a foliated complete surface $(Z, \cH)$  with only canonical singularities and proper birational morphisms $f: (Z, \cH)\to (X, \cF) $ and  $g: (Z, \cH)\to (Y, \cG)$. Therefore the required assertions follow by applying Theorem \ref{strong-vanishing} to the morphisms $f$ and $g$.
\end{proof}

\begin{Remark}
When $X$ is smooth, the above Corollary is proven in \cite[Theorem 3.1.1]{Men98}. 
\end{Remark}
\begin{Remark}
Let us note that for $m\ne 1$, $\chi (X, \cO_X(mK_{\cF}))$ is not a birational invariant of foliations with canonical singularities (not even for $m=0$).
\end{Remark}

\medskip

\begin{proof}[Proof of Theorem \ref{crepant-vanishing}.]
The proof is similar to that of Theorem \ref{strong-vanishing}. By Theorem \ref{strong-vanishing} the required assertion
is equivalent to vanishing of $\chi (y, \cO_X ( mK_{\cF}))=0$ for $m\ne 0$.
Let  $g: (Z, \cH)\to (X, \cF)$ be the minimal partial crepant resolution of singularities of $(X, \cF)$ (see Section \ref{special-minimal}).  By \S \ref{mod-Euler} we have
$$\chi (y, \cO_Z (m K_{\cH}))= \chi (y,  \cO_Y (mK_{\cF}))+\sum _{x\in f^{-1}(y)} \chi (x, \cO_Z (mK_{\cH})) .$$
Since all the numbers are non-negative, it is sufficient to prove that $\chi (x, \cO_Z ( mK_{\cH}))=0$ for $m\ne 0$.

We claim that $f\circ g$ factors into a composition of  the minimal partial crepant resolution of singularities
of $(Y, \cG)$ and blow ups at smooth points of the surface that are not regular for the foliation.
At such points $m(y)$ from the proof of Lemma  \ref{smooth-vanishing} is equal to zero and the same proof as that of Lemma  \ref{smooth-vanishing} shows that at such points $R^1\tilde f_*\cO (mK_{\tilde \cF})=0$ for $m\ne 0$. Therefore the required assertion follows from Theorem
\ref{strange-vanishing}.

To prove the claim let us first remark that $f$ is an isomorphism over points $y$ at which $\cG$ is terminal.
Indeed, this follows immediately from the fact that  for every prime divisor $E$ over such $y$ we have $a_E(\cG)>0$
and hence  if $f$ is not an isomorphism over $y$ we get a contradiction with $K_{\cF}= f^*K_{\cG}$. Since the claim is local on $Y$, we can therefore assume that $Y$ does not contain any singular points  at which $\cG$ is terminal.
In this case  we consider the minimal resolution of singularities $h: Z'\to Z$. Then $f\circ g\circ h$
can be factored  as  $f'\circ g'\circ h'$, where $f': T \to Y$ is  the minimal partial crepant resolution of singularities
of $(Y, \cG)$ and $g': T'\to T$ is the minimal  resolution of $T$. 
Let us note that 
$$K_{h^*\cH}=h^*K_{\cH}+\frac{1}{2}\sum E_i= (f\circ g \circ h)^*K_{\cG}+\frac{1}{2}\sum E_i,$$
where $\{E_i\}$ are disjoint curves with self intersection $-2$ (these curves arise when resolving singularities in Cases 2 and 3 in the proof of Theorem \ref{strange-vanishing}). We can also write 
$$ K_{(f'\circ g')^*\cG}= (f'\circ g')^*K_{\cG}+\frac{1}{2}\sum E_i',$$
where $\{E_i'\}$ are disjoint curves with self intersection $-2$. It follows that 
$$\sum E_i = (h')^*\sum E_i'.$$
Therefore $(h')^*\sum E_i'$ contains no  curves of self intersection $-1$ and hence $h'$ does not blow up any points lying on $\{E_i'\}$ and hence we have an induced morphism $Z\to T'$, which finishes proof of the claim and hence also of the theorem.
\end{proof}

\medskip

Theorem \ref{crepant-vanishing} together with Proposition \ref{sing-can-model} implies the following corollary.

\begin{Corollary}
The Hilbert function of a canonical model of a foliation determines the Hilbert function of any weak nef model.
More precisely, if $(X, \cF)$ is a weak nef model, $(Y, \cG)$ is a canonical model and $(X, \cF)$ and $(Y, \cG)$
are birationally equivalent, then
$$\chi (X, \cO _X(mK_{\cF}) )- \chi (Y, \cO _Y (mK_{\cG}) )= \left\{ 
\begin{array}{cl}
-c & \hbox{ if $m=0$,}\\
0& \hbox{ if $m\ne 0$,}\\
\end{array}
\right.
$$
where $c$ denotes the number of cusps of $Y$.
In particular, any two birationally equivalent  weak nef models have the same Hilbert function.
Similarly, any two birationally equivalent  canonical models have the same Hilbert function.
\end{Corollary}

\begin{proof}
To prove the required equality it is sufficient to compute $\chi (y, \cO _X(mK_{\cF}))$ at all points $y$ of $Y$ and apply the results of \S \ref{mod-Euler}.
\end{proof}

}

\medskip

\subsection*{Acknowledgements}

A large part of the paper was written during the second author's visit to Salt Lake City. He would like to thank
the University of Utah for excellent hiking and working conditions.
The authors would also like to thank the referees for their valuable comments and suggestions and Shou-Xian Li for pointing out an error in the proof of Lemma 2.5 in the previous version of this paper.

\end{document}